\documentclass[fleqn,12pt,twoside]{article}

\usepackage[headings]{espcrc1}
\readRCS
$Id: espcrc1.tex,v 1.2 2004/02/24 11:22:11 spepping Exp $
\usepackage{graphicx}
\usepackage[figuresright]{rotating}

\newtheorem{theorem}{Theorem}

\newtheorem{conjecture}[theorem]{Conjecture}
\newtheorem{corollary}[theorem]{Corollary}

\newtheorem{lemma}[theorem]{Lemma}

\newtheorem{problem}{Problem}

\newenvironment{proof}[1][Proof.]{\begin{trivlist}
\item[\hskip \labelsep {\bfseries #1}]}{\end{trivlist}}

\newenvironment{acknowledgement}[1][Acknowledgement]{\begin{trivlist}
\item[\hskip \labelsep {\bfseries #1}]}{\end{trivlist}}

\newcommand{\AmS}{{\protect\the\textfont2
  A\kern-.1667em\lower.5ex\hbox{M}\kern-.125emS}}

\usepackage{amsmath}
\usepackage{amsfonts}
\usepackage{amssymb}
\title{Interval cyclic edge-colorings of graphs}

\author{P.A. Petrosyan\address[MCSD]{Department of Informatics and Applied Mathematics,\\
Yerevan State University, 0025, Armenia}%
\address{Institute for Informatics and Automation Problems,\\
National Academy of Sciences, 0014, Armenia}%
\thanks {email: pet\_petros@ipia.sci.am},
        S.T. Mkhitaryan\addressmark[MCSD]%
\thanks{email: sargismk@ymail.com}}

\runtitle{Interval cyclic edge-colorings of graphs}\runauthor{P.A.
Petrosyan, S.T. Mkhitaryan}

\begin{document}

\maketitle

\begin{abstract}
A proper edge-coloring of a graph $G$ with colors $1,\ldots,t$ is
called an \emph{interval cyclic $t$-coloring} if all colors are
used, and the edges incident to each vertex $v\in V(G)$ are colored
by $d_{G}(v)$ consecutive colors modulo $t$, where $d_{G}(v)$ is the
degree of a vertex $v$ in $G$. A graph $G$ is \emph{interval
cyclically colorable} if it has an interval cyclic $t$-coloring for
some positive integer $t$. The set of all interval cyclically
colorable graphs is denoted by $\mathfrak{N}_{c}$. For a graph $G\in
\mathfrak{N}_{c}$, the least and the greatest values of $t$ for
which it has an interval cyclic $t$-coloring are denoted by
$w_{c}(G)$ and $W_{c}(G)$, respectively. In this paper we
investigate some properties of interval cyclic colorings. In
particular, we prove that if $G$ is a triangle-free graph with at
least two vertices and $G\in \mathfrak{N}_{c}$, then $W_{c}(G)\leq
\vert V(G)\vert +\Delta(G)-2$. We also obtain bounds on $w_{c}(G)$
and $W_{c}(G)$ for various classes of graphs. Finally, we
give some methods for constructing of interval cyclically non-colorable graphs.\\

\end{abstract}

\section{Introduction}\

All graphs considered in this paper are finite, undirected, and have
no loops or multiple edges. Let $V(G)$ and $E(G)$ denote the sets of
vertices and edges of $G$, respectively. For a graph $G$, the number
of connected components of $G$ is denoted by $c(G)$. A graph $G$ is
Eulerian if it has a closed trail containing every edge of $G$. The
degree of a vertex $v\in V(G)$ is denoted by $d_{G}(v)$ (or $d(v)$),
the maximum degree of $G$ by $\Delta(G)$, the diameter of $G$ by
$\mathrm{diam}(G)$, and the chromatic index of $G$ by
$\chi^{\prime}(G)$. The terms and concepts that we do not define can
be found in \cite{b3,b33}.

A proper edge-coloring of a graph $G$ with colors $1,\ldots ,t$ is
an \emph{interval $t$-coloring} if all colors are used, and the
colors of edges incident to each vertex of $G$ are form an interval
of integers. A graph $G$ is \emph{interval colorable} if it has an
interval $t$-coloring for some positive integer $t$. The concept of
interval edge-coloring of graphs was introduced by Asratian and
Kamalian \cite{b1}. In \cite{b1,b2}, the authors showed that if $G$
is interval colorable, then $\chi^{\prime}(G)=\Delta(G)$. In
\cite{b1,b2}, they also proved that if a triangle-free graph $G$ has
an interval $t$-coloring, then $t\leq \left\vert V(G)\right\vert
-1$. Later, Kamalian \cite{b13} showed that if $G$ admits an
interval $t$-coloring, then $t\leq 2\left\vert V(G)\right\vert -3$.
This upper bound was improved to $2\left\vert V(G)\right\vert -4$
for graphs $G$ with at least three vertices \cite{b8}. For an
$r$-regular graph $G$, Kamalian and Petrosyan \cite{b18} showed that
if $G$ with at least $2r+2$ vertices admits an interval
$t$-coloring, then $t\leq 2\left\vert V(G)\right\vert -5$. For a
planar graph $G$, Axenovich \cite{b4} showed that if $G$ has an
interval $t$-coloring, then $t\leq \frac{11}{6}\left\vert
V(G)\right\vert$. In \cite{b12,b13,b24,b26}, interval colorings of
complete graphs, complete bipartite graphs, trees, and
$n$-dimensional cubes were investigated. The $NP$-completeness of
the problem of the existence of an interval coloring of an arbitrary
bipartite graph was shown in \cite{b29}. In
\cite{b5,b6,b21,b25,b26,b27}, interval colorings of various products
of graphs were investigated. In
\cite{b2,b3,b6,b7,b9,b10,b11,b16,b17,b19}, the problem of the
existence and construction of interval colorings was considered, and
some bounds for the number of colors in such colorings of graphs
were given.

A proper edge-coloring of a graph $G$ with colors $1,\ldots,t$ is
called an \emph{interval cyclic $t$-coloring} if all colors are
used, and the edges incident to each vertex $v\in V(G)$ are colored
by $d_{G}(v)$ consecutive colors modulo $t$. A graph $G$ is
\emph{interval cyclically colorable} if it has an interval cyclic
$t$-coloring for some positive integer $t$. This type of
edge-coloring under the name of
\textquotedblleft$\pi$-coloring\textquotedblright was first
considered by Kotzig in \cite{b20}, where he proved that every cubic
graph has a $\pi$-coloring with 5 colors. However, the concept of
interval cyclic edge-coloring of graphs was explicitly introduced by
de Werra and Solot \cite{b32}. In \cite{b32}, they proved that if
$G$ is an outerplanar bipartite graph, then $G$ has an interval
cyclic $t$-coloring for any $t\geq \Delta(G)$. In \cite{b22}, Kubale
and Nadolski showed that the problem of determining whether a given
bipartite graph is interval cyclically colorable is $NP$-complete.
Later, Nadolski \cite{b23} showed that if $G$ is interval colorable,
then $G$ has an interval cyclic $\Delta(G)$-coloring. He also proved
that if $G$ is a connected graph with $\Delta(G)=3$, then $G$ has an
interval cyclic coloring with at most 4 colors. In \cite{b14,b15},
Kamalian investigated interval cyclic colorings of simple cycles and
trees. For simple cycles and trees, he determined all possible
values of $t$ for which these graphs have an interval cyclic
$t$-coloring.

In this paper we investigate some properties of interval cyclic
colorings. In particular, we prove that if a triangle-free graph $G$
with at least two vertices has an interval cyclic $t$-coloring, then
$t\leq \vert V(G)\vert +\Delta(G)-2$. For various classes of graphs,
we also obtain bounds on the least and the greatest values of $t$
for which these graphs have an interval cyclic $t$-coloring.
Finally, we describe some methods for constructing of interval
cyclically non-colorable graphs.

\section{Notations, definitions and auxiliary results}\

We use standard notations $C_{n},K_{n}$ and $Q_{n}$ for the simple
cycle, complete graph on $n$ vertices and the hypercube,
respectively. We also use standard notations $K_{m,n}$ and
$K_{l,m,n}$ for the complete bipartite and tripartite graph,
respectively, one part of which has $m$ vertices, other part has $n$
vertices and a third part has $l$ vertices.

A \emph{partial edge-coloring} of $G$ is a coloring of some of the
edges of $G$ such that no two adjacent edges receive the same color.
If $\alpha $ is a partial edge-coloring of $G$ and $v\in V(G)$, then
$S\left(v,\alpha \right)$ denotes the set of colors appearing on
colored edges incident to $v$.

A graph $G$ is \emph{interval colorable} if it has an interval
$t$-coloring for some positive integer $t$. The set of all interval
colorable graphs is denoted by $\mathfrak{N}$. For a graph $G\in
\mathfrak{N}$, the least and the greatest values of $t$ for which it
has an interval $t$-coloring are denoted by $w(G)$ and $W(G)$,
respectively.

A graph $G$ is \emph{interval cyclically colorable} if it has an
interval cyclic $t$-coloring for some positive integer $t$. The set
of all interval cyclically colorable graphs is denoted by
$\mathfrak{N}_{c}$. For a graph $G\in \mathfrak{N}_{c}$, the least
and the greatest values of $t$ for which it has an interval cyclic
$t$-coloring are denoted by $w_{c}(G)$ and $W_{c}(G)$, respectively.
The \emph{feasible set} $F(G)$ of a graph $G$ is the set of all
$t$'s such that there exists an interval cyclic $t$-coloring of $G$.
The feasible set of $G$ is \emph{gap-free} if
$F(G)=\left[w_{c}(G),W_{c}(G)\right]$. Clearly, if $G\in
\mathfrak{N}$, then $G\in \mathfrak{N}_{c}$ and
$\chi^{\prime}(G)\leq w_{c}(G)\leq w(G)\leq W(G)\leq W_{c}(G)\leq
\vert E(G)\vert$.

Let $\left\lfloor a\right\rfloor$ denote the largest integer less
than or equal to $a$. For two positive integers $a$ and $b$ with
$a\leq b$, we denote by $\left[a,b\right]$ the interval of integers
$\left\{a,\ldots,b\right\}$. By $\left[a,b\right]_{even}$
($\left[a,b\right]_{odd}$), we denote the set of all even (odd)
numbers from the interval $\left[a,b\right]$.\\

In \cite{b1,b2}, Asratian and Kamalian obtained the following two
results.

\begin{theorem}
\label{mytheorem1} If $G\in \mathfrak{N}$, then
$\chi^{\prime}(G)=\Delta(G)$. Moreover, if $G$ is a regular graph,
then $G\in \mathfrak{N}$ if and only if
$\chi^{\prime}(G)=\Delta(G)$.
\end{theorem}

\begin{theorem}
\label{mytheorem2} If $G$ is a connected triangle-free graph and
$G\in \mathfrak{N}$, then
\begin{center}
$W(G)\leq \vert V(G)\vert -1$.
\end{center}
\end{theorem}

For general graphs, Kamalian proved the following

\begin{theorem}
\label{mytheorem3}(\cite{b13}). If $G$ is a connected graph with at
least two vertices and $G\in \mathfrak{N}$, then
\begin{center}
$W(G)\leq 2\vert V(G)\vert -3$.
\end{center}
\end{theorem}

Note that the upper bound in Theorem \ref{mytheorem3} is sharp for
$K_{2}$, but if $G\neq K_{2}$, then this upper bound can be
improved.

\begin{theorem}
\label{mytheorem4}(\cite{b8}). If $G$ is a connected graph with with
at least three vertices and $G\in \mathfrak{N}$, then
\begin{center}
$W(G)\leq 2\vert V(G)\vert -4$.
\end{center}
\end{theorem}

In \cite{b30}, Vizing proved the following well-known result.

\begin{theorem}
\label{mytheorem5} For every graph $G$,
\begin{center}
$\Delta(G)\leq \chi^{\prime}(G)\leq \Delta(G)+1$.
\end{center}
\end{theorem}

\begin{corollary}
\label{mycorollary1} If $G$ is a regular graph, then $G\in
\mathfrak{N}_{c}$ and $w_{c}(G)=\chi^{\prime}(G)$.
\end{corollary}

From Theorems \ref{mytheorem1} and \ref{mytheorem5}, we get

\begin{corollary}
\label{mycorollary2} $\mathfrak{N}\subset \mathfrak{N}_{c}$.
\end{corollary}

Although all regular graphs are interval cyclically colorable, there
are many graphs that have no interval cyclic coloring. In Fig.
\ref{int-cyc-non graph}, we present the smallest known interval
cyclically non-colorable graph.

\begin{figure}[h]
\begin{center}
\includegraphics[width=10pc,height=17pc]{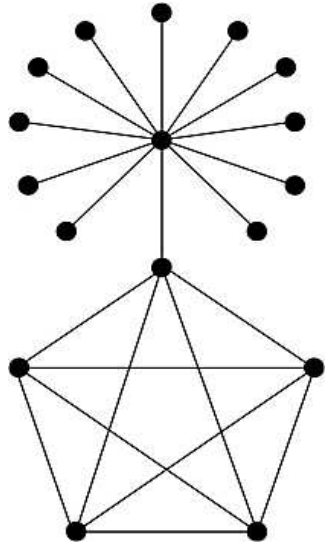}\\
\caption{The interval cyclically non-colorable
graph.}\label{int-cyc-non graph}
\end{center}
\end{figure}

We also need the generalizations of Theorems
\ref{mytheorem2},\ref{mytheorem3} and \ref{mytheorem4} for
disconnected graphs. It can be easily proved by induction on the
number of connected components that the following two lemmas hold.

\begin{lemma}
\label{mylemma1} If $G$ is a triangle-free graph and $G\in
\mathfrak{N}_{c}$, then
\begin{center}
$W(G)\leq \vert V(G)\vert -c(G)$.
\end{center}
\end{lemma}

\begin{lemma}
\label{mylemma2} If $G$ is a graph with at least two vertices and
$G\in \mathfrak{N}_{c}$, then
\begin{center}
$W(G)\leq 2\vert V(G)\vert -3\cdot c(G)$.
\end{center}
Moreover, if $G$ has at least three vertices, then
\begin{center}
$W(G)\leq 2\vert V(G)\vert -4\cdot c(G)$.
\end{center}
\end{lemma}\

\bigskip

\section{Some general results}\label{part1}\

In this section we derive some upper bounds for $W_{c}(G)$ depending
on the number of vertices, degrees and diameter for connected
graphs, triangle-free graphs, and, in particular, for bipartite
graphs. Next we show that there are graphs $G$ for which
$w_{c}(G)>\chi^{\prime}(G)$. We also investigate the feasible sets
of interval cyclically colorable graphs. In particular, we show that
if $G$ is interval colorable, then
$\left[\Delta(G),W(G)\right]\subseteq F(G)$. On the other hand, we
also show that there are interval cyclically colorable graphs for
which feasible sets are not gap-free.\\

Our first two theorems give upper bounds for $W_{c}(G)$ depending on
the number of vertices and the maximum degree of the interval
cyclically colorable graph $G$.

\begin{theorem}
\label{mytheorem6} If $G$ is a connected triangle-free graph with at
least two vertices and $G\in \mathfrak{N}_{c}$, then $W_{c}(G)\leq
\vert V(G)\vert +\Delta(G)-2$.
\end{theorem}
\begin{proof}
Consider an interval cyclic $W_{c}(G)$-coloring $\alpha$ of $G$. If
for each $u\in V(G)$, $S(u,\alpha)$ is an interval of integers, then
$G\in \mathfrak{N}$ and $W_{c}(G)\leq W(G)\leq \vert V(G)\vert-1
\leq \vert V(G)\vert +\Delta(G)-2$, by Theorem \ref{mytheorem2} and
taking into account that $G$ is a connected triangle-free graph with
at least two vertices.

Now suppose that there exists $v_{0}\in V(G)$ such that
$S(v_{0},\alpha)$ is not an interval of integers. Since $\alpha$ is
an interval cyclic $W_{c}(G)$-coloring of $G$, for each $v\in V(G)$
such that $S(v,\alpha)$ is not an interval, there are colors $k_{v}$
and $l_{v}$ such that

\begin{center}
$S(v,\alpha)=\{1,\ldots,k_{v}\}\cup
\{W_{c}(G)-l_{v}+1,\ldots,W_{c}(G)\}$.
\end{center}

Let $l^{\star}=\max_{\begin{subarray}{1}
v\in V(G),\\
S(v,\alpha)~is~not~an~interval \end{subarray}} l_{v}$. Clearly,
$1\leq l^{\star}\leq \Delta(G)-1$. Define an auxiliary graph $H$ as
follows:

\begin{center}
$V(H)=V(G)$ and\\
$E(H)=E(G)\setminus \{e\colon\, e\in E(G)\wedge \alpha(e)\in
\{W_{c}(G)-l^{\star}+1,\ldots,W_{c}(G)\}\}$.
\end{center}

Clearly, $H$ is a spanning subgraph of $G$. Now let us consider the
restriction of the coloring $\alpha$ on the edges of subgraph $H$ of
$G$. Let $\alpha_{H}$ be this edge-coloring. It is easy to see that
$\alpha_{H}$ is an interval $(W_{c}(G)-l^{\star})$-coloring of $H$.
Moreover, since $H$ is a triangle-free graph and $H\in
\mathfrak{N}$, by Lemma \ref{mylemma1}, we have

\begin{center}
$W_{c}(G)-l^{\star}\leq W(H)\leq \vert V(H)\vert-c(H)\leq \vert
V(H)\vert-1=\vert V(G)\vert-1$.
\end{center}

This implies that $W_{c}(G)\leq \vert V(G)\vert+l^{\star}-1$. From
this and taking into account that $l^{\star}\leq \Delta(G)-1$, we
obtain $W_{c}(G)\leq \vert V(G)\vert +\Delta(G)-2$. ~$\square$
\end{proof}

\begin{corollary}
\label{mycorollary3} If $G$ is a connected triangle-free graph with
at least two vertices and $G\in \mathfrak{N}_{c}$, then
$W_{c}(G)\leq 2\vert V(G)\vert -3$. Moreover, if $G$ has at least
three vertices, then $W_{c}(G)\leq 2\vert V(G)\vert -4$.
\end{corollary}

Note that the upper bound in Theorem \ref{mytheorem6} is sharp for
simple cycles, since $W_{c}(C_{n})=n$.

\begin{theorem}
\label{mytheorem7} If $G$ is a connected graph with at least two
vertices and $G\in \mathfrak{N}_{c}$, then $W_{c}(G)\leq 2\vert
V(G)\vert +\Delta(G)-4$. Moreover, if $G$ has at least three
vertices, then $W_{c}(G)\leq 2\vert V(G)\vert +\Delta(G)-5$.
\end{theorem}
\begin{proof}
Consider an interval cyclic $W_{c}(G)$-coloring $\alpha$ of $G$. If
for each $u\in V(G)$, $S(u,\alpha)$ is an interval of integers, then
$G\in \mathfrak{N}$ and $W_{c}(G)\leq W(G)\leq 2\vert V(G)\vert-3
\leq 2\vert V(G)\vert +\Delta(G)-4$, by Theorem \ref{mytheorem3} and
taking into account that $G$ is a connected graph with at least two
vertices. Moreover, if $G$ has at least three vertices, then
$W_{c}(G)\leq W(G)\leq 2\vert V(G)\vert-4 \leq 2\vert V(G)\vert
+\Delta(G)-5$, by Theorem \ref{mytheorem4}.

Now suppose that there exists $v_{0}\in V(G)$ such that
$S(v_{0},\alpha)$ is not an interval of integers. Since $\alpha$ is
an interval cyclic $W_{c}(G)$-coloring of $G$, for each $v\in V(G)$
such that $S(v,\alpha)$ is not an interval, there are colors $k_{v}$
and $l_{v}$ such that

\begin{center}
$S(v,\alpha)=\{1,\ldots,k_{v}\}\cup
\{W_{c}(G)-l_{v}+1,\ldots,W_{c}(G)\}$.
\end{center}

Let $l^{\star}=\max_{\begin{subarray}{1}
v\in V(G),\\
S(v,\alpha)~is~not~an~interval \end{subarray}} l_{v}$. Clearly,
$1\leq l^{\star}\leq \Delta(G)-1$. Define an auxiliary graph $H$ as
follows:

\begin{center}
$V(H)=V(G)$ and\\
$E(H)=E(G)\setminus \{e\colon\, e\in E(G)\wedge \alpha(e)\in
\{W_{c}(G)-l^{\star}+1,\ldots,W_{c}(G)\}\}$.
\end{center}

Clearly, $H$ is a spanning subgraph of $G$. Now let us consider the
restriction of the coloring $\alpha$ on the edges of the subgraph
$H$ of $G$. Let $\alpha_{H}$ be this edge-coloring. It is easy to
see that $\alpha_{H}$ is an interval $(W_{c}(G)-l^{\star})$-coloring
of $H$. Since $H\in \mathfrak{N}$, by Lemma \ref{mylemma2}, we have

\begin{center}
$W_{c}(G)-l^{\star}\leq W(H)\leq 2\vert V(H)\vert-3\cdot c(H)\leq
2\vert V(H)\vert-3=2\vert V(G)\vert-3$.
\end{center}

Moreover, if $G$ has at least three vertices, then

\begin{center}
$W_{c}(G)-l^{\star}\leq W(H)\leq 2\vert V(H)\vert-4\cdot c(H)\leq
2\vert V(H)\vert-4=2\vert V(G)\vert-4$.
\end{center}

This implies that $W_{c}(G)\leq 2\vert V(G)\vert+l^{\star}-3$. From
this and taking into account that $l^{\star}\leq \Delta(G)-1$, we
obtain $W_{c}(G)\leq 2\vert V(G)\vert +\Delta(G)-4$. Moreover, if
$G$ has at least three vertices, then $W_{c}(G)\leq 2\vert V(G)\vert
+\Delta(G)-5$. ~$\square$
\end{proof}

\begin{corollary}
\label{mycorollary4} If $G$ is a connected graph with at least two
vertices and $G\in \mathfrak{N}_{c}$, then $W_{c}(G)\leq 3\vert
V(G)\vert -5$. Moreover, if $G$ has at least three vertices, then
$W_{c}(G)\leq 3\vert V(G)\vert -6$.
\end{corollary}

Note that the first upper bound in Theorem \ref{mytheorem7} is sharp
for $K_{2}$ and the second upper bound is sharp for $K_{3}$, but we
strongly believe that these upper bounds can be improved. Next we
give some upper bounds for $W_{c}(G)$ depending on degrees and
diameter of the interval cyclically colorable connected graph $G$.

\begin{theorem}
\label{mytheorem8} If $G$ is a connected graph and $G\in
\mathfrak{N}_{c}$, then
\begin{center}
$W_{c}(G)\leq 1+2\cdot{\max\limits_{P\in
\mathbf{P}}}{\sum\limits_{v\in V(P)}}\left(d_{G}(v)-1\right)$,
\end{center}
where $\mathbf{P}$ is the set of all shortest paths in the graph
$G$.
\end{theorem}

\begin{proof} Consider an interval cyclic $W_{c}(G)$-coloring
$\alpha $ of $G$. Let us show that $W_{c}(G)\leq
1+2\cdot{\max\limits_{P\in \mathbf{P}}}{\sum\limits_{v\in
V(P)}}\left(d_{G}(v)-1\right)$. Suppose, to the contrary, that
$W_{c}(G)> 1+2\cdot{\max\limits_{P\in \mathbf{P}}}{\sum\limits_{v\in
V(P)}}\left(d_{G}(v)-1\right)$. In the coloring $\alpha$ of $G$, we
consider the edges with colors $1$ and $2+{\max\limits_{P\in
\mathbf{P}}}{\sum\limits_{v\in V(P)}}\left(d_{G}(v)-1\right)$. Let
$e=u_{1}u_{2}, e^{\prime}=w_{1}w_{2}$ and $\alpha(e)=1,
\alpha(e^{\prime})=2+{\max\limits_{P\in
\mathbf{P}}}{\sum\limits_{v\in V(P)}}\left(d_{G}(v)-1\right)$.
Without loss of generality we may assume that a shortest path $P$
joining $e$ with $e^{\prime}$ joins $u_{1}$ with $w_{1}$, where

\begin{center}
$P=\left(v_{0},e_{1},v_{1},\ldots ,v_{i-1},e_{i},v_{i},\ldots
,v_{k-1},e_{k},v_{k}\right)$ and $v_{0}=u_{1}$, $v_{k}=w_{1}$.
\end{center}

Since $\alpha$ is an interval cyclic coloring of $G$, we have

\begin{center}
either $\alpha(e_{1})\leq d_{G}(v_{0})$ or $\alpha(e_{1})\geq
W_{c}(G)-d_{G}(v_{0})+2$,

either $\alpha(e_{2})\leq \alpha(e_{1})+d_{G}(v_{1})-1$ or
$\alpha(e_{2})\geq \alpha(e_{1})-d_{G}(v_{1})+1$,

$\cdots \cdots \cdots \cdots \cdots \cdots$

either $\alpha(e_{i})\leq \alpha(e_{i-1})+d_{G}(v_{i-1})-1$ or
$\alpha(e_{i})\geq \alpha(e_{i-1})-d_{G}(v_{i-1})+1$,

$\cdots \cdots \cdots \cdots \cdots \cdots$

either $\alpha(e_{k})\leq \alpha(e_{k-1})+d_{G}(v_{k-1})-1$ or
$\alpha(e_{k})\geq \alpha(e_{k-1})-d_{G}(v_{k-1})+1$.
\end{center}

Summing up these inequalities, we obtain

\begin{center}
either $\alpha(e_{k})\leq
1+{\sum\limits_{j=0}^{k-1}\left(d_{G}(v_{j})-1\right)}$ or
$\alpha(e_{k})\geq
W_{c}(G)+1-{\sum\limits_{j=0}^{k-1}\left(d_{G}(v_{j})-1\right)}$.
\end{center}

Hence, we have either
\begin{equation}\label{eq:1}
\alpha(e^{\prime})\leq \alpha(e_{k})+d_{G}(v_{k})-1\leq
1+{\sum\limits_{j=0}^{k}\left(d_{G}(v_{j})-1\right)}
\end{equation}
or
\begin{equation}\label{eq:2}
\alpha(e^{\prime})\geq \alpha(e_{k})-d_{G}(v_{k})+1\geq
W_{c}(G)+1-{\sum\limits_{j=0}^{k}\left(d_{G}(v_{j})-1\right)}.
\end{equation}

On the other hand, by (\ref{eq:1}), we obtain

$2+{\max\limits_{P\in \mathbf{P}}}{\sum\limits_{v\in
V(P)}}\left(d_{G}(v)-1\right)=\alpha(e^{\prime})\leq
1+{\sum\limits_{j=0}^{k}\left(d_{G}(v_{j})-1\right)}\leq
1+{\max\limits_{P\in \mathbf{P}}}{\sum\limits_{v\in
V(P)}}\left(d_{G}(v)-1\right)$, which is a contradiction.\\

Similarly, by (\ref{eq:2}), we obtain

$2+{\max\limits_{P\in \mathbf{P}}}{\sum\limits_{v\in
V(P)}}\left(d_{G}(v)-1\right)=\alpha(e^{\prime})\geq
W_{c}(G)+1-{\sum\limits_{j=0}^{k}\left(d_{G}(v_{j})-1\right)}\geq
W_{c}(G)+1-{\max\limits_{P\in \mathbf{P}}}{\sum\limits_{v\in
V(P)}}\left(d_{G}(v)-1\right)$ and thus $W_{c}(G)\leq
1+2\cdot{\max\limits_{P\in \mathbf{P}}}{\sum\limits_{v\in
V(P)}}\left(d_{G}(v)-1\right)$, which is a contradiction. ~$\square$
\end{proof}

\begin{corollary}
\label{mycorollary5}(\cite{b23}). If $G$ is a connected graph and
$G\in \mathfrak{N}_{c}$, then
\begin{center}
$W_{c}(G)\leq 1+2(\mathrm{diam}(G)+1)\left(\Delta(G)-1\right)$.
\end{center}
\end{corollary}

\begin{theorem}
\label{mytheorem9} If $G$ is a connected bipartite graph and $G\in
\mathfrak{N}_{c}$, then
\begin{center}
$W_{c}(G)\leq 1+2\cdot \mathrm{diam}(G)\left(\Delta(G)-1\right)$.
\end{center}
\end{theorem}
\begin{proof} Consider an interval cyclic
$W_{c}(G)$-coloring $\alpha $ of $G$. Let us show that $W_{c}(G)\leq
1+2\cdot \mathrm{diam}(G)\left(\Delta(G)-1\right)$. Suppose, to the
contrary, that $W_{c}(G)>1+2\cdot
\mathrm{diam}(G)\left(\Delta(G)-1\right)$. In the coloring $\alpha$
of $G$, we consider the edges with colors $1$ and
$2+\mathrm{diam}(G)\left(\Delta(G)-1\right)$. Let $e=u_{1}u_{2},
e^{\prime}=w_{1}w_{2}$ and $\alpha(e)=1,
\alpha(e^{\prime})=2+\mathrm{diam}(G)\left(\Delta(G)-1\right)$.
Since for any two edges in a bipartite graph $G$, some two of their
endpoints must be at a distance of at most $\mathrm{diam}(G)-1$ from
each other, we may assume that there is the path $P$ joining $e$ and
$e^{\prime}$ with the length is not greater than
$\mathrm{diam}(G)-1$. Also, we may assume that $P$ joining $e$ with
$e^{\prime}$ joins $u_{1}$ with $w_{1}$, where

\begin{center}
$P=\left(v_{0},e_{1},v_{1},\ldots ,v_{i-1},e_{i},v_{i},\ldots
,v_{k-1},e_{k},v_{k}\right)$ and $v_{0}=u_{1}$, $v_{k}=w_{1}$.
\end{center}

Since $\alpha$ is an interval cyclic coloring of $G$, for $1\leq
i\leq k$, we have

\begin{center}
either $\alpha(v_{i-1}v_{i})\leq
1+{\sum\limits_{j=0}^{i-1}\left(d_{G}(v_{j})-1\right)}$ or
$\alpha(v_{i-1}v_{i})\geq
W_{c}(G)+1-{\sum\limits_{j=0}^{i-1}\left(d_{G}(v_{j})-1\right)}$.
\end{center}

From this, we have either
\begin{equation}\label{eq:3}
\alpha(e^{\prime})\leq
1+{\sum\limits_{j=0}^{k}\left(d_{G}(v_{j})-1\right)}
\end{equation}
or
\begin{equation}\label{eq:4}
\alpha(e^{\prime})\geq
W_{c}(G)+1-{\sum\limits_{j=0}^{k}\left(d_{G}(v_{j})-1\right)}.
\end{equation}

On the other hand, by (\ref{eq:3}) and taking into account that
$k\leq \mathrm{diam}(G)-1$, we obtain

$2+\mathrm{diam}(G)\left(\Delta(G)-1\right)=\alpha(e^{\prime})\leq
1+{\sum\limits_{j=0}^{k}\left(d_{G}(v_{j})-1\right)}\leq
1+\mathrm{diam}(G)\left(\Delta(G)-1\right)$, which is a contradiction.\\

Similarly, by (\ref{eq:4}) and taking into account that $k\leq
\mathrm{diam}(G)-1$, we obtain

$2+\mathrm{diam}(G)\left(\Delta(G)-1\right)=\alpha(e^{\prime})\geq
W_{c}(G)+1-{\sum\limits_{j=0}^{k}\left(d_{G}(v_{j})-1\right)}\geq
W_{c}(G)+1-\mathrm{diam}(G)\left(\Delta(G)-1\right)$ and thus
$W_{c}(G)\leq 1+2\cdot \mathrm{diam}(G)\left(\Delta(G)-1\right)$,
which is a contradiction. ~$\square$
\end{proof}

Now we show that the coefficient 2 in the last upper bounds cannot
be improved.

\begin{theorem}
\label{mytheorem10} For any integers $d\geq 2$ and $n\geq 3$, there
exists a connected graph $G$ with $\Delta(G)=d$ and
$\mathrm{diam}(G)=\left\lfloor\frac{n}{2}\right\rfloor+2$ such that
$G\in \mathfrak{N}_{c}$ and $W_{c}(G)=n(d-1)$.
\end{theorem}
\begin{proof} For the proof, we construct a graph $G_{d,n}$ that satisfies the
specified conditions. We define a graph $G_{d,n}$ as follows:
\begin{center}
$V(G_{d,n})=\{v_{1},\ldots,v_{n}\}\cup
\left\{u^{(i)}_{j}\colon\,1\leq i\leq
n,1\leq j\leq d-2\right\}$ and\\
$E(G_{d,n})=\{v_{i}v_{i+1}\colon\, 1\leq i\leq
n-1\}\cup\left\{v_{1}v_{n}\right\}\cup\left\{v_{i}u^{(i)}_{j}\colon\,1\leq
i\leq n,1\leq j\leq d-2\right\}$.
\end{center}

Clearly, $G_{d,n}$ is a connected graph with $\Delta(G_{d,n})=d$ and
$\mathrm{diam}(G_{d,n})=\left\lfloor\frac{n}{2}\right\rfloor+2$.

Let us show that $G_{d,n}$ has an interval cyclic $n(d-1)$-coloring.

Define an edge-coloring $\alpha$ of $G_{d,n}$ as follows:
\begin{description}
\item[(1)] for $1\leq i\leq n$ and $1\leq j\leq d-2$, let
\begin{center}
$\alpha\left(v_{i}u^{(i)}_{j}\right)=(i-1)(d-1)+j$;
\end{center}

\item[(2)] for $1\leq i\leq n-1$, let
\begin{center}
$\alpha\left(v_{i}v_{i+1}\right)=i(d-1)$ and
$\alpha\left(v_{1}v_{n}\right)=n(d-1)$.
\end{center}
\end{description}

It is easy to see that $\alpha$ is an interval cyclic
$n(d-1)$-coloring of $G_{d,n}$. This implies that $G_{d,n}\in
\mathfrak{N}_{c}$ and $W_{c}(G_{d,n})\geq n(d-1)$. On the other
hand, clearly $W_{c}(G_{d,n})\leq \vert E(G_{d,n})\vert=n(d-1)$ and
thus $W_{c}(G_{d,n})=n(d-1)$. ~$\square$
\end{proof}\

In the last part of the section we investigate the feasible sets of
interval cyclically colorable graphs.

\begin{theorem}
\label{mytheorem11} If $G\in\mathfrak{N}$, then $G\in
\mathfrak{N}_{c}$ and $\left[\Delta(G),W(G)\right]\subseteq F(G)$.
\end{theorem}
\begin{proof} Since any interval $t$-coloring of $G$ is also
an interval cyclic $t$-coloring of $G$, we obtain that $G\in
\mathfrak{N}_{c}$.

Assume that $\Delta(G)\leq t\leq W(G)$. Let $\alpha$ be an interval
$W(G)$-coloring of $G$. Define an edge-coloring $\beta$ of $G$ as
follows: for every $e\in E(G)$, let
\begin{center}
$\beta(e)=\left\{
\begin{tabular}{ll}
$\alpha(e)\pmod{t}$, & if $\alpha(e)\pmod{t}\neq 0$,\\
$t$, & otherwise.\\
\end{tabular}%
\right.$
\end{center}

It is easy to see that $\beta$ is an interval cyclic $t$-coloring of
$G$. ~$\square$
\end{proof}

\begin{corollary}
\label{mycorollary6}(\cite{b23}). If $G\in\mathfrak{N}$, then $G\in
\mathfrak{N}_{c}$ and $w_{c}(G)=\Delta(G)$.
\end{corollary}

\begin{theorem}
\label{mytheorem12} If $G$ is an Eulerian graph and $\vert
E(G)\vert$ is odd, then $G$ has no interval cyclic $t$-coloring for
every even positive integer $t$.
\end{theorem}
\begin{proof}
Suppose, to the contrary, that $G$ has an interval cyclic
$t$-coloring $\alpha$ for some even positive integer $t$. Since $G$
is an Eulerian graph, $G$ is connected and $d_{G}(v)$ is even for
any $v\in V(G)$, by Euler's Theorem. Since $\alpha$ is an interval
cyclic coloring and all degrees of vertices of $G$ are even, we have
that for any $v\in V(G)$, the set $S\left(v,\alpha\right)$ contains
exactly $\frac{d_{G}(v)}{2}$ even colors and $\frac{d_{G}(v)}{2}$
odd colors. Now let $m_{odd}$ be the number of edges with odd colors
in the coloring $\alpha$. By Handshaking lemma, we obtain
$m_{odd}=\frac{1}{2}\sum\limits_{v\in
V(G)}\frac{d_{G}(v)}{2}=\frac{\vert E(G)\vert}{2}$. Thus $\vert
E(G)\vert$ is even, which is a contradiction. ~$\square$
\end{proof}

\begin{corollary}
\label{mycorollary7} If $G$ is an interval cyclically colorable
Eulerian graph with an odd number of edges and
$\chi^{\prime}(G)=\Delta(G)$, then $w_{c}(G)>\chi^{\prime}(G)$.
\end{corollary}

\begin{figure}[h]
\begin{center}
\includegraphics[width=25pc,height=10pc]{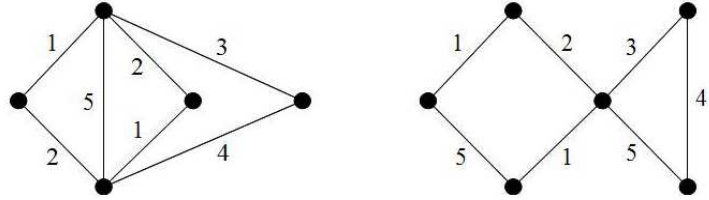}\\
\caption{Interval cyclic $5$-colorings of $K_{1,1,3}$ and
$F$.}\label{K_{1,1,3} and Fish graph}
\end{center}
\end{figure}

Fig. \ref{K_{1,1,3} and Fish graph} shows the complete tripartite
graph $K_{1,1,3}$ and the fish graph $F$ that are smallest interval
cyclically colorable Eulerian graphs with an odd number of edges and
for which the chromatic index is equal to the maximum degree.
Clearly, $\chi^{\prime}(K_{1,1,3})=\chi^{\prime}(F)=4$, but
$w_{c}(K_{1,1,3})=w_{c}(F)=5$.

Finally let us consider the simple path $P_{m}$ and the simple cycle
$C_{n}$ ($n\geq 3$). Clearly, $P_{m}\in\mathfrak{N}_{c}$ and
$w_{c}(P_{m})=\Delta(P_{m})$, $W_{c}(P_{m})=m-1$. Moreover,
$F(P_{m})=\left[w_{c}(P_{m}),W_{c}(P_{m})\right]$, so the feasible
set of $P_{m}$ is gap-free. Now let us consider the simple cycle
$C_{n}$. Clearly, $C_{n}\in\mathfrak{N}_{c}$ and
$w_{c}(C_{n})=\chi^{\prime}(C_{n})$, $W_{c}(C_{n})=n$. Moreover, it
is not difficult to see that any simple cycle with an odd number of
vertices has an interval cyclic $t$-coloring for every odd integer
$t$, $3\leq t\leq n$, so, by Theorem \ref{mytheorem12}, we obtain

\begin{corollary}
\label{mycorollary8} For any odd integer $n\geq 3$, we have
$F(C_{n})=\left[3,n\right]_{odd}$.
\end{corollary}

This corollary implies that for any odd integer $n\geq 5$, the
feasible set of $C_{n}$ is not gap-free. A more general result on
the feasible set of simple cycles was obtained by Kamalian in
\cite{b15}.

\begin{theorem}
\label{mytheorem13}(\cite{b15}). For any integer $n\geq 3$, we have
\begin{center}
$F(C_{n})=\left\{
\begin{tabular}{ll}
$\left[3,n\right]_{odd}$, & if $n$ is odd,\\
$\left[2,\frac{n}{2}+1\right]\cup \left[\frac{n}{2}+2,n\right]_{even}$, & if $n=4k,k\in\mathbb{N}$,\\
$\left[2,\frac{n}{2}+1\right]\cup \left[\frac{n}{2}+3,n\right]_{even}$, & if $n=4k+2,k\in\mathbb{N}$.\\
\end{tabular}%
\right.$
\end{center}
\end{theorem}\

\section{Interval cyclic edge-colorings of complete
graphs}\label{part2}\

This section is devoted to interval cyclic colorings of complete
graphs. In \cite{b31}, Vizing proved the following

\begin{theorem}
\label{mytheorem14} For the complete graph $K_{n}$ ($n\geq 2$), we
have
\begin{center}
$\chi^{\prime}(K_{n})=\left\{
\begin{tabular}{ll}
$n-1$, & if $n$ is even, \\
$n$, & if $n$ is odd. \\
\end{tabular}%
\right.$
\end{center}
\end{theorem}

From Corollary \ref{mycorollary1} and Theorem \ref{mytheorem14}, we
obtain that if $n\in \mathbb{N}$, then $K_{2n},K_{2n+1}\in
\mathfrak{N}_{c}$ and $w_{c}(K_{2n})=2n-1,w_{c}(K_{2n+1})=2n+1$. Now
let us consider the parameters $W_{c}(K_{2n})$ and $W_{c}(K_{2n+1})$
when $n\in \mathbb{N}$. In \cite{b24}, Petrosyan investigated
interval colorings of complete graphs and hypercubes. In particular,
he proved the following

\begin{theorem}
\label{mytheorem15} If $n=p2^{q}$, where $p$ is odd and $q$ is
nonnegative, then
\begin{center}
$W(K_{2n})\geq 4n-2-p-q$.
\end{center}
Moreover, if $2n-1\leq t\leq 4n-2-p-q$, then $K_{2n}$ has an
interval $t$-coloring.
\end{theorem}

\begin{corollary}
\label{mycorollary9} If $n=p2^{q}$, where $p$ is odd and $q$ is
nonnegative, then
\begin{center}
$W_{c}(K_{2n})\geq 4n-2-p-q$.
\end{center}
Moreover, $\left[2n-1,4n-2-p-q\right]\subseteq F(K_{2n})$.
\end{corollary}

\begin{figure}[h]
\begin{center}
\includegraphics[width=35pc,height=21pc]{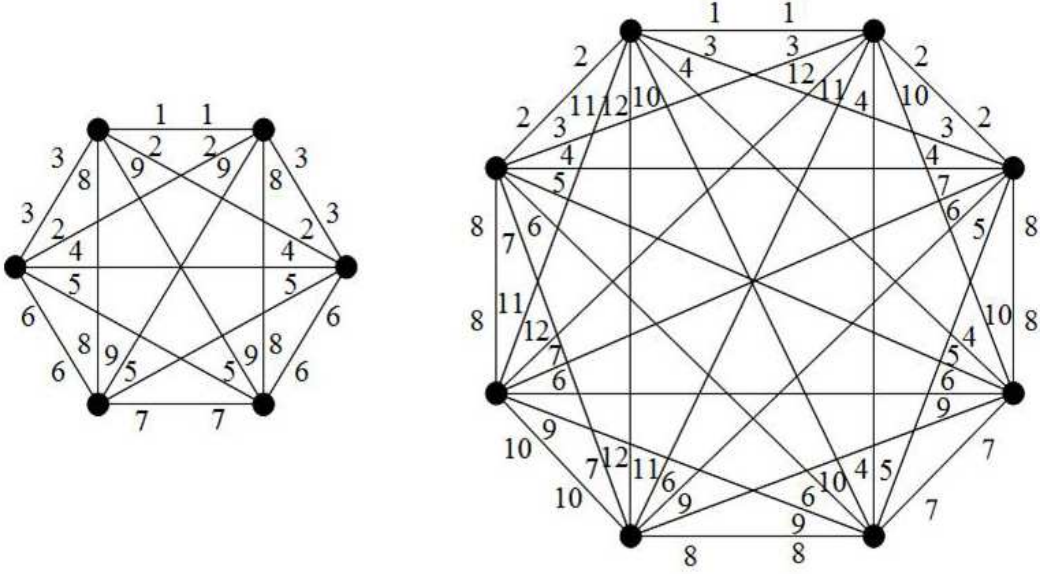}\\
\caption{The interval cyclic $9$-coloring of $K_{6}$ and the
interval cyclic $12$-coloring of $K_{8}$.}\label{K_{6} and K_{8}}
\end{center}
\end{figure}

On the other hand, Corollary \ref{mycorollary4} implies that
$W_{c}(K_{2n})\leq 6n-6$ for $n\geq 2$. It is not difficult to see
that $W(K_{4})=W_{c}(K_{4})=4$. In \cite{b24}, it was proved that
$W(K_{6})=7$ and $W(K_{8})=11$, but Fig. \ref{K_{6} and K_{8}} shows
that $W_{c}(K_{6})\geq 9$ and $W(K_{8})\geq 12$, so
$W_{c}(K_{6})>W(K_{6})$ and $W_{c}(K_{8})>W(K_{8})$. In general, we
strongly believe that $W_{c}(K_{2n})>W(K_{2n})$ for $n\geq 3$. Now
we give a lower bound for $W_{c}(K_{2n+1})$ when $n\in \mathbb{N}$.

\begin{theorem}
\label{mytheorem16} If $n\in \mathbb{N}$, then $W_{c}(K_{2n+1})\geq
3n$.
\end{theorem}
\begin{proof} For the proof, it suffices to construct an interval
cyclic $3n$-coloring of $K_{2n+1}$. Let
$V(K_{2n+1})=\left\{v_{0},v_{1},\ldots,v_{2n}\right\}$.

Define an edge-coloring $\beta$ of $K_{2n+1}$. For each edge
$v_{i}v_{j}\in E(K_{2n+1})$ with $i<j$, define a color
$\beta\left(v_{i}v_{j}\right)$ as follows:\\

$\beta\left(v_{i}v_{j}\right)=\left\{
\begin{tabular}{ll}
$1$, & if $i=0$, $j=1$;\\
$2n+1$, & if $i=0$, $j=2$;\\
$j-1$, & if $i=0$, $3\leq j\leq n$;\\
$n+1+j$, & if $i=0$, $n+1\leq j\leq 2n-2$;\\
$n$, & if $i=0$, $j=2n-1$;\\
$3n$, & if $i=0$, $j=2n$;\\
$i+j-1$, & if $1\leq i\leq \left\lfloor\frac{n}{2}\right\rfloor$,
$2\leq j\leq n$, $i+j\leq n+1$;\\
$i+j+n-2$, & if $2\leq i\leq n-1$, $\left\lfloor
\frac{n}{2}\right\rfloor +2\leq j\leq n$, $i+j\geq n+2$;\\
$n+1+j-i$, & if $3\leq i\leq n$, $n+1\leq j\leq 2n-2$, $j-i\leq
n-2$;\\
$j-i+1$, & if $1\leq i\leq n$, $n+1\leq j\leq 2n$, $j-i\geq n$;\\
$2i-1$, & if $2\leq i\leq 1+\left\lfloor
\frac{n-1}{2}\right\rfloor$, $n+1\leq j\leq n+\left\lfloor
\frac{n-1}{2}\right\rfloor$, $j-i=n-1$;\\
$i+j-1$, & if $\left\lfloor \frac{n-1}{2}\right\rfloor +2\leq i\leq
n$, $n+1+\left\lfloor \frac{n-1}{2}\right\rfloor\leq j\leq 2n-1$,
$j-i=n-1$;\\
$i+j-2n+1$, & if $n+1\leq i\leq n+\left\lfloor
\frac{n}{2}\right\rfloor -1$, $n+2\leq j\leq 2n-2$, $i+j\leq
3n-1$;\\
$i+j-n$, & if $n+1\leq i\leq 2n-1$, $n+\left\lfloor
\frac{n}{2}\right\rfloor +1\leq j\leq 2n$, $i+j\geq 3n$.
\end{tabular}%
\right.$\\

Let us prove that $\beta $ is an interval cyclic $3n$-coloring of
$K_{2n+1}$.

Let $G$ be the subgraph of $K_{2n+1}$ induced by
$\{v_{1},\ldots,v_{2n}\}$. Clearly, $G$ is isomorphic to $K_{2n}$.
This edge-coloring $\beta$ of $K_{2n+1}$ is constructed and based on
the interval $(3n-2)$-coloring of $G$ which is described in the
proof of Theorem 4 from \cite{b24}. We use this interval
$(3n-2)$-coloring of $G$ and then we shift all colors of the edges
of $G$ by one. Let $\alpha$ be this edge-coloring of $G$. Using the
property of this edge-coloring which is described in the proof of
Corollary 6 from \cite{b24}, we get

\begin{description}
\item[1)] $S\left(v_{1},\alpha\right)=S\left(v_{2},\alpha\right)=[2,2n]$,

\item[2)] $S\left(v_{i},\alpha\right)=S\left(v_{n+i-2},\alpha\right)=[i,2n-2+i]$ for $3\leq i\leq n$,

\item[3)]
$S\left(v_{2n-1},\alpha\right)=S\left(v_{2n},\alpha\right)=[n+1,3n-1]$.
\end{description}

Now, by the definition of $\beta$, we have

\begin{description}
\item[1)] $S\left(v_{0},\beta\right)=[1,n]\cup [2n+1,3n]$,

\item[2)] $S\left(v_{1},\beta\right)=[1,2n]$ and $S\left(v_{2},\beta\right)=[2,2n+1]$,

\item[3)] $S\left(v_{i},\beta\right)=[i-1,2n-2+i]$ and $S\left(v_{n+i-2},\beta\right)=[i,2n-1+i]$ for $3\leq i\leq n$,

\item[4)]
$S\left(v_{2n-1},\beta\right)=[n,3n-1]$ and
$S\left(v_{2n},\beta\right)=[n+1,3n]$.
\end{description}

This shows that $\beta $ is an interval cyclic $3n$-coloring of
$K_{2n+1}$ and hence $W_{c}(K_{2n+1})\geq 3n$. ~$\square$
\end{proof}

Note that the lower bound in Theorem \ref{mytheorem16} is sharp for
$K_{3}$, since $W_{c}(K_{3})=3$. On the other hand, Corollary
\ref{mycorollary4} implies that $W_{c}(K_{2n+1})\leq 6n-3$ for
$n\in\mathbb{N}$. It is worth also noting that, in general, the
feasible set of $K_{2n+1}$ is not gap-free. For example, $K_{7}$ has
an interval cyclic $7$-coloring and, by Theorem \ref{mytheorem16},
it also has an interval cyclic $9$-coloring, but since $\vert
E\left(K_{7}\right)\vert=21$, Theorem \ref{mytheorem12} implies that
$K_{7}$ has no interval cyclic $8$-coloring, so $F(K_{7})$ is not
gap-free.

\section{Interval cyclic edge-colorings of complete bipartite and tripartite
graphs}\label{part3}\

In this section we show that all complete bipartite and tripartite
graphs are interval cyclically colorable. We also obtain some bounds
for parameters $W_{c}\left(K_{m,n}\right)$ and
$w_{c}\left(K_{l,m,n}\right)$ when $l,m,n\in \mathbb{N}$. In
\cite{b12}, Kamalian investigated interval colorings of complete
bipartite graphs and trees. In particular, he proved the following

\begin{theorem}
\label{mytheorem17} For any $m,n\in \mathbb{N}$, we have
\begin{description}
\item[(1)] $K_{m,n}\in \mathfrak{N}$,

\item[(2)] $w\left(K_{m,n}\right)=m+n-\gcd(m,n)$,

\item[(3)] $W\left(K_{m,n}\right)=m+n-1$,

\item[(4)] if $w\left(K_{m,n}\right)\leq t\leq W\left(K_{m,n}\right)$, then $K_{m,n}$
has an interval $t$-coloring.
\end{description}
\end{theorem}

We first prove the theorem on the feasible set of complete bipartite
graphs.

\begin{theorem}
\label{mytheorem18} If $\min\{m,n\}=1$, then
$w_{c}(K_{m,n})=W_{c}(K_{m,n})=m+n-1$. If $\min\{m,n\}\geq 2$, then
$\left[\max\{m,n\},m+n\right]\subseteq F\left(K_{m,n}\right)$.
\end{theorem}
\begin{proof} First note that if $\min\{m,n\}=1$, then $K_{m,n}$ is
a star and hence $w_{c}(K_{m,n})=W_{c}(K_{m,n})=m+n-1$.

Assume that $\min\{m,n\}\geq 2$. Let us show that if
$\max\{m,n\}\leq t\leq m+n$, then $K_{m,n}$ has an interval cyclic
$t$-coloring. By Theorems \ref{mytheorem11} and \ref{mytheorem17},
we have $\left[\max\{m,n\},m+n-1\right]\subseteq
F\left(K_{m,n}\right)$. Now we prove that $K_{m,n}$ has an interval
cyclic $(m+n)$-coloring.

Let $V(K_{m,n})=\{u_{1},\ldots,u_{m},v_{1},\ldots,v_{n}\}$ and
$E(K_{m,n})=\left\{u_{i}v_{j}\colon\,1\leq i\leq m,1\leq j\leq
n\right\}$.

Define an edge-coloring $\alpha$ of $K_{m,n}$ as follows: for $1\leq
i\leq m$ and $1\leq j\leq n$, let

\begin{center}
$\alpha\left(u_{i}v_{j}\right)=\left\{
\begin{tabular}{ll}
$i+j-1$, & if $(i,j)\neq (1,n)$,\\
$m+n$, & otherwise.\\
\end{tabular}%
\right.$
\end{center}

It is not difficult to see that $\alpha$ is an interval cyclic
$(m+n)$-coloring of $K_{m,n}$ and hence
$\left[\max\{m,n\},m+n\right]\subseteq F\left(K_{m,n}\right)$ when
$\min\{m,n\}\geq 2$. ~$\square$
\end{proof}

\begin{corollary}
\label{mycorollary10} If $\min\{m,n\}=1$, then
$w_{c}(K_{m,n})=W_{c}(K_{m,n})=m+n-1$. If $\min\{m,n\}\geq 2$, then
$w_{c}(K_{m,n})=\max\{m,n\}$ and $W_{c}(K_{m,n})\geq m+n$.
\end{corollary}

Now we show that all complete tripartite graphs are interval
cyclically colorable.

\begin{theorem}
\label{mytheorem19} For any $l,m,n\in \mathbb{N}$, we have
$K_{l,m,n}\in \mathfrak{N}_{c}$ and $w_{c}(K_{l,m,n})\leq l+m+n$.
\end{theorem}
\begin{proof} Without loss of generality we may assume that $l\leq m\leq
n$. Clearly, for the proof, it suffices to construct an interval
cyclic $(l+m+n)$-coloring of $K_{l,m,n}$.

Let
$V(K_{l,m,n})=\{u_{1},\ldots,u_{m},v_{1},\ldots,v_{n},w_{1},\ldots,w_{l}\}$
and $E(K_{l,m,n})=\left\{u_{i}v_{j}\colon\,1\leq i\leq m,1\leq j\leq
n\right\}\cup \left\{u_{i}w_{j}\colon\,1\leq i\leq m,1\leq j\leq
l\right\}\cup \left\{w_{i}v_{j}\colon\,1\leq i\leq l,1\leq j\leq
n\right\}$.

Define an edge-coloring $\alpha$ of $K_{l,m,n}$ as follows:

\begin{description}
\item[(1)] for $1\leq i\leq m$ and $1\leq j\leq n$, let
\begin{center}
$\alpha \left(u_{i}v_{j}\right)=l+i+j-1$;
\end{center}

\item[(2)] for $1\leq i\leq l$ and $1\leq j\leq n$, let
\begin{center}
$\alpha \left(w_{i}v_{j}\right)=i+j-1$;
\end{center}

\item[(3)] for $1\leq i\leq m$, $1\leq j\leq l$ and $i+j\leq m+1$, let
\begin{center}
$\alpha \left(u_{i}w_{j}\right)=l+n+i+j-1$;
\end{center}

\item[(4)] for $1\leq i\leq m$, $1\leq j\leq l$ and $i+j\geq m+2$, let
\begin{center}
$\alpha \left(u_{i}w_{j}\right)=i+j-m-1$.
\end{center}
\end{description}

Let us prove that $\alpha$ is an interval cyclic $(l+m+n)$-coloring
of $K_{l,m,n}$.

By the definition of $\alpha$, we have
\begin{description}
\item[1)] for $1\leq i\leq m-l+1$,

$S\left(u_{i},\alpha\right)=[l+i,l+n+i-1]\cup
[l+n+i,2l+n+i-1]=[l+i,2l+n+i-1]$ due to (1) and (3),

\item[2)] for $m-l+2\leq i\leq m$,

$S\left(u_{i},\alpha\right)=[l+i,l+n+i-1]\cup [l+n+i,l+m+n]\cup
[1,l-m-1+i]=[1,l-m-1+i]\cup [l+i,l+m+n]$ due to (1),(3) and (4),

\item[3)] for $1\leq i\leq n$,

$S\left(v_{i},\alpha\right)=[l+i,l+m+i-1]\cup [i,l+i-1]=[i,l+m+i-1]$
due to (1) and (2),

\item[4)]for $1\leq i\leq l$,

$S\left(w_{i},\alpha\right)=[i,n+i-1]\cup [l+n+i,l+m+n]\cup
[1,i-1]=[1,n+i-1]\cup [l+n+i,l+m+n]$ due to (2),(3) and (4).
\end{description}

This implies that $\alpha$ is an interval cyclic $(l+m+n)$-coloring
of $K_{l,m,n}$; thus $K_{l,m,n}\in \mathfrak{N}_{c}$ and
$w_{c}\left(K_{l,m,n}\right)\leq l+m+n$ for $l,m,n\in \mathbb{N}$.
~$\square$
\end{proof}

\begin{corollary}
\label{mycorollary11} For any $l,m,n\in \mathbb{N}$, we have
$K_{l,m,n}\in \mathfrak{N}_{c}$ and $W_{c}(K_{l,m,n})\geq l+m+n$.
\end{corollary}

Note that the upper bound in Theorem \ref{mytheorem19} is sharp for
$K_{1,m,n}$ when $m$ and $n$ are odd, since
$w_{c}\left(K_{1,1,1}\right)=\chi^{\prime}\left(C_{3}\right)=3$ and
for $\max\{m,n\}\geq 3$, $m+n+1\geq
w_{c}\left(K_{1,m,n}\right)>\chi^{\prime}\left(K_{1,m,n}\right)=\Delta\left(K_{1,m,n}\right)=m+n$
due to Corollary \ref{mycorollary7}. It is worth also noting that
although all complete tripartite graphs are interval cyclically
colorable, in \cite{b9} Grzesik and Khachatrian proved that
$K_{1,m,n}$ is interval colorable if and only if $\gcd(m+1,n+1)=1$.
This implies that there are infinitely many complete tripartite
graphs from the class $\mathfrak{N}_{c}\setminus \mathfrak{N}$.

\section{Interval cyclic edge-colorings of hypercubes}\label{part4}\

In this section we show that hypercubes $Q_{n}$ are interval
cyclically colorable. We also obtain some bounds for the parameter
$W_{c}\left(Q_{n}\right)$ when $n\in \mathbb{N}$. In \cite{b24},
Petrosyan investigated interval colorings of complete graphs and
hypercubes. In particular, he proved that $Q_{n}\in \mathfrak{N}$
and $w(Q_{n})=n$, $W(Q_{n})\geq\frac{n(n+1)}{2}$ for any
$n\in\mathbf{N}$. In the same paper he also conjectured that
$W(Q_{n})=\frac{n(n+1)}{2}$ for any $n\in\mathbf{N}$. In \cite{b26},
the authors confirmed this conjecture. This implies that $Q_{n}\in
\mathfrak{N}_{c}$ and $w_{c}(Q_{n})=n$,
$W_{c}(Q_{n})\geq\frac{n(n+1)}{2}$. Moreover, by Theorem
\ref{mytheorem11}, we obtain
$\left[n,\frac{n(n+1)}{2}\right]\subseteq F(Q_{n})$. On the other
hand, since $Q_{n}$ is a connected bipartite graph and taking into
account that $\mathrm{diam}\left(Q_{n}\right)=\Delta
\left(Q_{n}\right)=n$, we get $W_{c}\left(Q_{n}\right)\leq 1+2\cdot
\mathrm{diam}\left(Q_{n}\right)\left(\Delta\left(Q_{n}\right)-1\right)=2n^{2}-2n+1$,
by Theorem \ref{mytheorem9}. So, we have $\frac{n(n+1)}{2}\leq
W_{c}\left(Q_{n}\right)\leq 2n^{2}-2n+1$ for any $n\in\mathbf{N}$.
Now we prove a new lower bound for $W_{c}\left(Q_{n}\right)$ which
improves $W_{c}(Q_{n})\geq\frac{n(n+1)}{2}$ for $2\leq n\leq 5$.

\begin{figure}[h]
\begin{center}
\includegraphics[width=15pc,height=13pc]{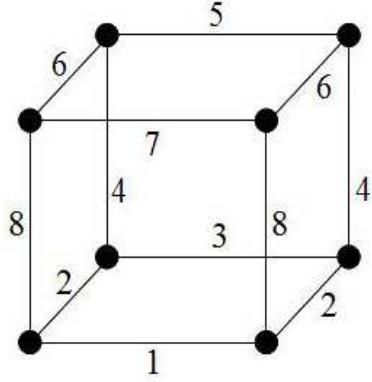}\\
\caption{The interval cyclic $8$-coloring of $Q_{3}$.}\label{Q_{3}}
\end{center}
\end{figure}

\begin{theorem}
\label{mytheorem20} For any integer $n\geq 2$, we have
$W_{c}\left(Q_{n}\right)\geq 4(n-1)$.
\end{theorem}
\begin{proof} First let us note that for any
integer $n\geq 2$, $Q_{n}$ has an interval $(n+1)$-coloring such
that for one half of vertices of $Q_{n}$, the set of colors
appearing on edges incident to these vertices is an interval $[1,n]$
and for remaining half of vertices of $Q_{n}$, the set of colors
appearing on edges incident to these remaining vertices is an
interval $[2,n+1]$. It can be easily done by induction on $n$. Also,
it is not difficult to see that
$W_{c}\left(Q_{2}\right)=W_{c}\left(C_{4}\right)=4$ and
$W_{c}\left(Q_{3}\right)\geq 8$ (See Fig. \ref{Q_{3}}).

Assume that $n\geq 4$.

For $(i,j)\in \{0,1\}^{2}$, let $Q_{n-2}^{(i,j)}$ be the subgraph of
$Q_{n}$ induced by the vertices

\begin{center}
$\left\{\left(i,j,\alpha _{3},\ldots ,\alpha _{n}\right)\colon\,
\left(\alpha _{3},\ldots ,\alpha _{n}\right)\in
\left\{0,1\right\}^{n-2}\right\}$.
\end{center}

Each $Q_{n-2}^{(i,j)}$ is isomorphic to $Q_{n-2}$. Let $\varphi$ be
an interval $(n-1)$-coloring of $Q_{n-2}^{(0,0)}$ such that for one
half of vertices of $Q_{n-2}^{(0,0)}$, the set of colors appearing
on edges incident to these vertices be an interval $[1,n-2]$ and for
remaining half of vertices of $Q_{n-2}^{(0,0)}$, the set of colors
appearing on edges incident to these remaining vertices be an
interval $[2,n-1]$.

Let us define an edge-coloring $\psi $ of subgraphs
$Q_{n-2}^{(0,1)}$, $Q_{n-2}^{(1,1)}$ and $Q_{n-2}^{(1,0)}$ of
$Q_{n}$ as follows:

\begin{description}
\item[(1)] for every edge $\left(0,1,\bar
\alpha\right)\left(0,1,\bar\beta\right)\in
E\left(Q_{n-2}^{(0,1)}\right)$, let

\begin{center}
$\psi\left(\left(0,1,\bar
\alpha\right)\left(0,1,\bar\beta\right)\right) =\varphi
\left(\left(0,0,\bar
\alpha\right)\left(0,0,\bar\beta\right)\right)+n-1$;
\end{center}

\item[(2)] for every edge $\left(1,1,\bar
\alpha\right)\left(1,1,\bar\beta\right)\in
E\left(Q_{n-2}^{(1,1)}\right)$, let

\begin{center}
$\psi\left(\left(1,1,\bar
\alpha\right)\left(1,1,\bar\beta\right)\right) =\varphi
\left(\left(0,0,\bar
\alpha\right)\left(0,0,\bar\beta\right)\right)+2n-2$;
\end{center}

\item[(3)] for every edge $\left(1,0,\bar
\alpha\right)\left(1,0,\bar\beta\right)\in
E\left(Q_{n-2}^{(1,0)}\right)$, let

\begin{center}
$\psi\left(\left(1,0,\bar
\alpha\right)\left(1,0,\bar\beta\right)\right) =\varphi
\left(\left(0,0,\bar
\alpha\right)\left(0,0,\bar\beta\right)\right)+3n-3$.
\end{center}
\end{description}

Now we define an edge-coloring $\lambda$ of the graph $Q_{n}$.

For every edge $\tilde{\alpha}\tilde{\beta} \in E\left(
Q_{n}\right)$, let

\begin{center}
$\lambda \left(\tilde{\alpha}\tilde{\beta}\right)=\left\{
\begin{tabular}{ll}
$\varphi\left(\tilde{\alpha}\tilde{\beta}\right),$ & if $%
\tilde{\alpha},\tilde{\beta}\in V\left(Q_{n-2}^{(0,0)}\right)$,\\
$\psi\left(\tilde{\alpha}\tilde{\beta}\right),$ & if $%
\tilde{\alpha},\tilde{\beta}\in V\left(Q_{n-2}^{(0,1)}\right)$ or
$\tilde{\alpha},\tilde{\beta}\in V\left(Q_{n-2}^{(1,1)}\right)$ or
$\tilde{\alpha},\tilde{\beta}\in V\left(Q_{n-2}^{(1,0)}\right)$,\\
$n-1,$ & if $\tilde{\alpha}\in V\left(Q_{n-2}^{(0,0)}\right)$,
$\tilde{\beta}\in
V\left(Q_{n-2}^{(0,1)}\right)$, $S\left(\tilde{\alpha},\varphi\right)=[1,n-2]$,\\
$n,$ & if $\tilde{\alpha}\in V\left(Q_{n-2}^{(0,0)}\right)$,
$\tilde{\beta}\in
V\left(Q_{n-2}^{(0,1)}\right)$, $S\left(\tilde{\alpha},\varphi\right)=[2,n-1]$,\\
$2n-2,$ & if $\tilde{\alpha}\in V\left(Q_{n-2}^{(0,1)}\right)$,
$\tilde{\beta}\in
V\left(Q_{n-2}^{(1,1)}\right)$, $S\left(\tilde{\alpha},\psi\right)=[n,2n-3]$,\\
$2n-1,$ & if $\tilde{\alpha}\in V\left(Q_{n-2}^{(0,1)}\right)$,
$\tilde{\beta}\in
V\left(Q_{n-2}^{(1,1)}\right)$, $S\left(\tilde{\alpha},\psi\right)=[n+1,2n-2]$,\\
$3n-3,$ & if $\tilde{\alpha}\in V\left(Q_{n-2}^{(1,1)}\right)$,
$\tilde{\beta}\in
V\left(Q_{n-2}^{(1,0)}\right)$, $S\left(\tilde{\alpha},\psi\right)=[2n-1,3n-4]$,\\
$3n-2,$ & if $\tilde{\alpha}\in V\left(Q_{n-2}^{(1,1)}\right)$,
$\tilde{\beta}\in
V\left(Q_{n-2}^{(1,0)}\right)$, $S\left(\tilde{\alpha},\psi\right)=[2n,3n-3]$,\\
$4n-4,$ & if $\tilde{\alpha}\in V\left(Q_{n-2}^{(1,0)}\right)$,
$\tilde{\beta}\in
V\left(Q_{n-2}^{(0,0)}\right)$, $S\left(\tilde{\alpha},\psi\right)=[3n-2,4n-5]$,\\
$1,$ & if $\tilde{\alpha}\in V\left(Q_{n-2}^{(1,0)}\right)$,
$\tilde{\beta}\in
V\left(Q_{n-2}^{(0,0)}\right)$, $S\left(\tilde{\alpha},\psi\right)=[3n-1,4n-4]$.%
\end{tabular}%
\right.$
\end{center}

It is easy to verify that $\lambda$ is an interval cyclic
$(4n-4)$-coloring of $Q_{n}$; thus $W_{c}\left(Q_{n}\right)\geq
4(n-1)$ for $n\geq 2$. ~$\square$
\end{proof}

This theorem implies that $W_{c}\left(Q_{2}\right)=4,
W_{c}\left(Q_{3}\right)\geq 8, W_{c}\left(Q_{4}\right)\geq 12$ and
$W_{c}\left(Q_{5}\right)\geq 16$. Moreover, it is not difficult to
see that $F\left(Q_{2}\right)=[2,4]$, $F\left(Q_{3}\right)=[3,8]$,
$[4,12]\subseteq F\left(Q_{4}\right)$ and $[5,16]\subseteq
F\left(Q_{5}\right)$. We strongly believe that the feasible set of
$Q_{n}$ is gap-free, but this is an open problem.

\section{Graphs that have no interval cyclic
edge-coloring}\label{part5}\

In this section we describe two methods for constructing of interval
cyclically non-colorable graphs. Our first method is based on trees
and it was earlier used for constructing of interval non-colorable
graphs in \cite{b28}.

Let $T$ be a tree and $V(T)=\{v_{1},\ldots,v_{n}\}$, $n\geq 2$. Let
$P(v_{i},v_{j})$ be a simple path joining $v_{i}$ and $v_{j}$ in
$T$, $VP(v_{i},v_{j})$ and $EP(v_{i},v_{j})$ denote the sets of
vertices and edges of this path, respectively. Also, let
$L(T)=\{v\colon\,v\in V(T)\wedge d_{T}(v)=1\}$. For a simple path
$P(v_{i},v_{j})$, define $LP(v_{i},v_{j})$ as follows:
\begin{center}
$LP(v_{i},v_{j})=\vert EP(v_{i},v_{j})\vert
+\vert\left\{vw\colon\,vw\in E(T), v\in VP(v_{i},v_{j}), w\notin
VP(v_{i},v_{j})\right\}\vert$.
\end{center}

Define: $M(T)={\max }_{1\leq i<j\leq n}LP(v_{i},v_{j})$. Let us
define the graph $\widetilde{T}$ as follows:
\begin{center}
$V(\widetilde{T})=V(T)\cup \{u\}$, $u\notin V(T)$,
$E(\widetilde{T})=E(T)\cup \{uv\colon\,v\in L(T)\}$.
\end{center}

Clearly, $\widetilde{T}$ is a connected graph with
$\Delta(\widetilde{T})=\vert L(T)\vert$. Moreover, if $T$ is a tree
in which the distance between any two pendant vertices is even, then
$\widetilde{T}$ is a connected bipartite graph.

In \cite{b14}, Kamalian proved the following result.

\begin{theorem}
\label{mytheorem20} If $T$ is a tree, then
\begin{description}
\item[(1)] $T\in \mathfrak{N}_{c}$,

\item[(2)] $w_{c}(T)=\Delta(G)$,

\item[(3)] $W_{c}(T)=M(T)$,

\item[(4)] $F(T)=[w_{c}(T),W_{c}(T)]$.
\end{description}
\end{theorem}

\begin{theorem}
\label{mytheorem21} If $T$ is a tree and $\vert L(T)\vert \geq
2(M(T)+2)$, then $\widetilde{T}\notin \mathfrak{N}_{c}$.
\end{theorem}
\begin{proof}
Suppose, to the contrary, that $\widetilde{T}$ has an interval
cyclic $t$-coloring $\alpha$ for some $t\geq \vert L(T)\vert$.

Consider the vertex $u$. Without loss of generality we may assume
that $S(u,\alpha)=[1,\vert L(T)\vert]$. Let $v$ and $v^{\prime}$ be
two vertices adjacent to $u$ such that $\alpha(uv)=1$ and
$\alpha(uv^{\prime})=M(T)+3$. Since $\widetilde{T}-u$ is a tree,
there is a unique path $P(v,v^{\prime})$ in $\widetilde{T}-u$
joining $v$ with $v^{\prime}$, where
\begin{center}
$P(v,v^{\prime})=(x_{0},e_{1},x_{1},\ldots,x_{i-1},e_{i},x_{i},\ldots,x_{k-1},e_{k},x_{k})$,
$x_{0}=v$, $x_{k}=v^{\prime}$.
\end{center}

Since $\alpha$ is an interval cyclic coloring of $G$, for $1\leq
i\leq k$, we have

\begin{center}
either $\alpha(x_{i-1}x_{i})\leq
2+{\sum\limits_{j=0}^{i-1}\left(d_{T}(x_{j})-1\right)}$ or
$\alpha(x_{i-1}x_{i})\geq
t-{\sum\limits_{j=0}^{i-1}\left(d_{T}(x_{j})-1\right)}$.
\end{center}

From this, we have either
\begin{equation}\label{eq:5}
\alpha(x_{k-1}x_{k})=\alpha(x_{k-1}v^{\prime})\leq
2+{\sum\limits_{j=0}^{k-1}\left(d_{T}(v_{j})-1\right)}=1+LP(v,v^{\prime})\leq
1+M(T)
\end{equation}
or
\begin{equation}\label{eq:6}
\alpha(x_{k-1}x_{k})=\alpha(x_{k-1}v^{\prime})\geq
t-{\sum\limits_{j=0}^{k-1}\left(d_{T}(v_{j})-1\right)}=t+1-LP(v,v^{\prime})\geq
t+1-M(T).
\end{equation}

On the other hand, by (\ref{eq:5}), we obtain
$M(T)+3=\alpha(uv^{\prime})\leq 2+M(T)$, which is a contradiction.
Similarly, by (\ref{eq:6}), we obtain
$M(T)+3=\alpha(uv^{\prime})\geq t-M(T)$ and thus $t\leq 2M(T) +3$,
which is a contradiction. ~$\square$
\end{proof}

\begin{corollary}
\label{mycorollary12} If $T$ is a tree in which the distance between
any two pendant vertices is even and $\vert L(T)\vert \geq
2(M(T)+2)$, then the bipartite graph $\widetilde{T}$ has no interval
cyclic coloring.
\end{corollary}

\begin{figure}[h]
\begin{center}
\includegraphics[width=40pc,height=7pc]{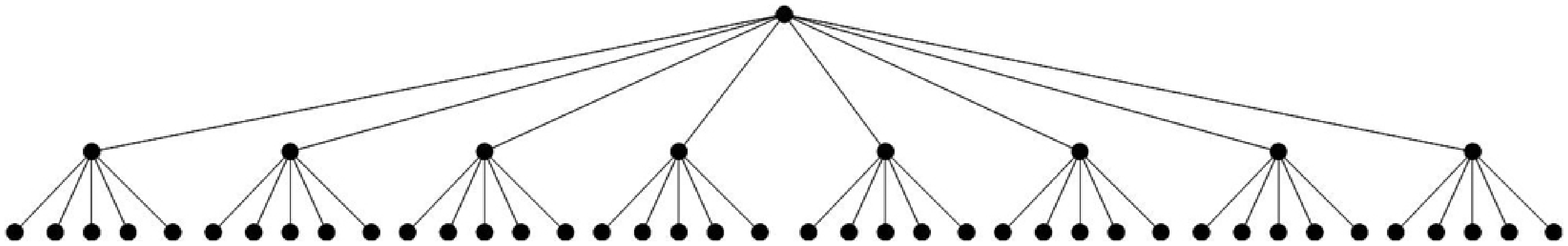}\\
\caption{The tree $T$.}\label{Tree}
\end{center}
\end{figure}

Now let us consider the tree $T$ shown in Fig. \ref{Tree}. Since
$M(T)=18$ and $\vert L(T)\vert=40$, the bipartite graph
$\widetilde{T}$ with $\vert V(\widetilde{T})\vert=50$ and
$\Delta(\widetilde{T})=40$ has no interval cyclic coloring.

The second method which we consider is based on complete graphs and
it was first described in \cite{b23}, but here we prove a more
stronger result.

Let $K_{2n+1}$ be a complete graph on $2n+1$ vertices and
$V(K_{2n+1})=\left\{v_{1},\ldots,v_{2n+1}\right\}$. For any $m,n\in
\mathbb{N}$, define the graph $K_{2n+1}^{\star m}$ as follows:
\begin{center}
$V(K_{2n+1}^{\star m})=V(K_{2n+1})\cup
\left\{u,w_{1},\ldots,w_{m}\right\}$, $E(K_{2n+1}^{\star
m})=E(K_{2n+1})\cup \left\{v_{1}u\right\}\cup
\left\{uw_{i}\colon\,1\leq i\leq m\right\}$.
\end{center}

Clearly, $K_{2n+1}^{\star m}$ is a connected with $\vert
V\left(K_{2n+1}^{\star m}\right)\vert=m+2n+2$ and
$\Delta\left(K_{2n+1}^{\star m}\right)=\max\{m+1,2n+1\}$.

\begin{theorem}
\label{mytheorem22} If $n\geq 2$ and $m\geq 6n$, then
$K_{2n+1}^{\star m}\notin \mathfrak{N}_{c}$.
\end{theorem}
\begin{proof}
Suppose, to the contrary, that $K_{2n+1}^{\star m}$ has an interval
cyclic $t$-coloring $\alpha$ for some $t\geq d(u)=6n+1$.

Let $H=K_{2n+1}^{\star m}-w_{1}-w_{2}-\cdots-w_{m}$. Also, let
$C=\bigcup_{v\in V(H)}S(v,\alpha)$ and $\overline C=[1,t]\setminus
C$. Since $H$ is connected, it is not difficult to see that either
$C$ or $\overline C$ is an interval of integers. Let $\vert C\vert =
t^{\prime}$. Clearly, $t^{\prime}\leq t$. Now let us consider the
restriction of the coloring $\alpha$ on the edges of the subgraph
$H$ of $K_{2n+1}^{\star m}$. Let $\alpha_{H}$ be this edge-coloring.
By rotating of colors of $C$ along the cycle with colors
$1,\ldots,t$, we get a new edge-coloring $\alpha_{H}^{\prime}$ of
$H$ with colors $1,\ldots,t^{\prime}$. Since $\alpha$ is an interval
cyclic $t$-coloring of $K_{2n+1}^{\star m}$ and taking into account
that the vertex $u$ in $H$ is pendant, we obtain that
$\alpha_{H}^{\prime}$ is an interval cyclic $t^{\prime}$-coloring of
$H$. Moreover, by Corollary \ref{mycorollary4}, we have
$t^{\prime}\leq 3\vert V(H)\vert -6=3(2n+2)-6=6n$. Since $t\geq
6n+1$, we get that $\alpha_{H}^{\prime}$ is also an interval
$t^{\prime}$-coloring of $H$. In \cite{b8}, it was proved that
$H\notin \mathfrak{N}$, so this contradiction proves the theorem.
~$\square$
\end{proof}

\begin{corollary}
\label{mycorollary13} For any integer $d\geq 13$, there exists a
connected graph $G$ such that $G\notin \mathfrak{N}_{c}$ and
$\Delta(G)=d$.
\end{corollary}

Now we show that $K_{5}^{\star 11}\notin \mathfrak{N}_{c}$. Note
that $\vert V\left(K_{5}^{\star 11}\right)\vert=17$ and
$\Delta\left(K_{5}^{\star 11}\right)=12$. Suppose that, to the
contrary, that $K_{5}^{\star 11}$ has an interval cyclic
$t$-coloring $\alpha$ for some $t\geq 12$. Similarly as in the proof
of Theorem \ref{mytheorem22}, we can consider the subgraph
$H=K_{5}^{\star 11}-w_{1}-w_{2}-\cdots-w_{11}$ of $K_{5}^{\star
11}$. Let $t^{\prime}=\left\vert\bigcup_{v\in
V(H)}S(v,\alpha)\right\vert$ and $\alpha_{H}$ be the restriction of
the coloring $\alpha$ on the edges of the subgraph $H$ of
$K_{5}^{\star 11}$. Then, let $\alpha_{H}^{\prime}$ be the
edge-coloring of $H$ with colors $1,\ldots,t^{\prime}$. Since
$\alpha$ is an interval cyclic $t$-coloring of $K_{5}^{\star 11}$
and taking into account that the vertex $u$ in $H$ is pendant, it
can be easily seen that $\alpha_{H}^{\prime}$ is an interval cyclic
$t^{\prime}$-coloring of $H$, where $t^{\prime}\leq t$. Clearly,
$t^{\prime}\leq \vert E(H)\vert=11$. Since $t\geq 12$, we obtain
that $\alpha_{H}^{\prime}$ is also an interval $t^{\prime}$-coloring
of $H$, which is a contradiction. From here, we get the following

\begin{corollary}
\label{mycorollary14} For any integer $d\geq 12$, there exists a
connected graph $G$ such that $G\notin \mathfrak{N}_{c}$ and
$\Delta(G)=d$.
\end{corollary}\

\section{Problems and Conjectures}\

In this section we collected different problems and conjectures that
arose in previous sections. Our first conjectures concern the
parameters $w_{c}(G)$ and $W_{c}(G)$ of an interval cyclically
colorable graph $G$. In section \ref{part1}, we proved that if $G$
is a connected triangle-free graph with at least two vertices and
$G\in \mathfrak{N}_{c}$, then $W_{c}(G)\leq \vert V(G)\vert
+\Delta(G)-2$. However, we think that the maximum degree in the
upper bound can be omitted; more precisely we believe that the
following is true:

\begin{conjecture}\label{myconjecture1}
If $G$ is a connected triangle-free graph and $G\in
\mathfrak{N}_{c}$, then $W_{c}(G)\leq \vert V(G)\vert$.
\end{conjecture}

Note that if Conjecture \ref{myconjecture1} is true, then this upper
bound cannot be improved, since $W_{c}(K_{m,n})\geq m+n$
($\min\{m,n\}\geq 2$), by Corollary \ref{mycorollary10}. We also
proved that if $G$ is a connected graph with at least two vertices
and $G\in \mathfrak{N}_{c}$, then $W_{c}(G)\leq 2\vert V(G)\vert
+\Delta(G)-4$. We again think that the maximum degree in this upper
bound can be omitted; more precisely we believe that the following
is true:

\begin{conjecture}\label{myconjecture2}
If $G$ is a connected graph with at least two vertices and $G\in
\mathfrak{N}_{c}$, then $W_{c}(G)\leq 2\vert V(G)\vert-3$.
\end{conjecture}

It is worth noting that there exists a connection between Conjecture
\ref{myconjecture1} and Conjecture \ref{myconjecture2}. If one can
prove that Conjecture \ref{myconjecture1} is true, then we able to
show that $W_{c}(G)\leq 2\vert V(G)\vert-1$ for an interval
cyclically colorable connected graph $G$.\\

It is known that all regular graphs $G$ are interval cyclically
colorable and $w_{c}(G)=\chi^{\prime}(G)$. Moreover, if $G$ is
interval colorable, then $G$ is interval cyclically colorable and
$w_{c}(G)=\chi^{\prime}(G)=\Delta(G)$. On the other hand, in section
\ref{part1}, it was shown that there are many interval cyclically
colorable graphs $G$ for which $w_{c}(G)>\chi^{\prime}(G)$. So, it
is interesting to investigate the following

\begin{problem}\label{myproblem1}
Characterize all interval cyclically colorable graphs $G$ for which
$w_{c}(G)=\chi^{\prime}(G)$.
\end{problem}

In section \ref{part1}, we also investigated the feasible sets of
interval cyclically colorable graphs. In particular, we proved that
if $G$ is interval colorable, then
$\left[\Delta(G),W(G)\right]\subseteq F(G)$. On the other hand, we
gave some examples of interval cyclically graphs $G$ for which
$F(G)$ is not gap-free. So, it is interesting to investigate the
following

\begin{problem}\label{myproblem2}
Characterize all interval cyclically colorable graphs $G$ for which
$F(G)$ is gap-free.
\end{problem}

For example, we know that if $T$ is a tree, then $F(T)$ is gap-free
\cite{b14}, but we also strongly believe that for any $m,n\in
\mathbb{N}$, $F(K_{2n})$, $F(K_{m,n})$ and $F(Q_{n})$ are
gap-free.\\

In sections \ref{part2} and \ref{part3}, we investigated interval
cyclic colorings of complete, complete bipartite and tripartite
graphs, but the following problems are still open:

\begin{problem}\label{myproblem3}
What is the exact value of $W_{c}\left(K_{n}\right)$ for any $n\in
\mathbb{N}$?
\end{problem}

\begin{problem}\label{myproblem4}
What are the exact values of $w_{c}\left(K_{l,m,n}\right)$ and
$W_{c}\left(K_{m,n}\right)$, $W_{c}\left(K_{l,m,n}\right)$ for any
$l,m,n\in \mathbb{N}$?
\end{problem}

In sections \ref{part2} and \ref{part3}, we proved that all complete
bipartite and tripartite graphs are interval cyclically colorable,
but we think that a more general result is true:

\begin{conjecture}\label{myconjecture3}
All complete multipartite graphs are interval cyclically colorable.
\end{conjecture}

In section \ref{part4}, we investigated interval cyclic colorings of
hypercubes $Q_{n}$ and proved that
$W_{c}\left(Q_{n}\right)=O(n^{2})$, but the following problem
remains open:

\begin{problem}\label{myproblem5}
What is the exact value of $W_{c}\left(Q_{n}\right)$ for any $n\in
\mathbb{N}$?
\end{problem}

In \cite{b23}, Nadolski showed that if $G$ is a connected graph with
$\Delta(G)=3$, then $G\in \mathfrak{N}_{c}$ and $w_{c}(G)\leq 4$.
From here and taking into account that all simple paths and cycles
are interval cyclically colorable, we obtain that all subcubic
graphs are interval cyclically colorable. On the other hand, in
section \ref{part5}, we proved that for any integer $d\geq 12$,
there exists a connected graph $G$ such that $G\notin
\mathfrak{N}_{c}$ and $\Delta(G)=d$. So, it is naturally to consider
the following

\begin{problem}\label{myproblem6}
Is there a connected graph $G$ such that $4\leq \Delta(G)\leq 11$
and $G\notin \mathfrak{N}_{c}$?
\end{problem}

\begin{acknowledgement}
We would like to thank the organizers of 7-th Cracow conference on
Graph Theory \textquotedblleft Rytro' 14\textquotedblright for the
nice environment and working atmosphere at the conference.
\end{acknowledgement}

\end{document}